\documentclass[]{interact}

\usepackage{epstopdf}% To incorporate .eps illustrations using PDFLaTeX, etc.
\usepackage[caption=false]{subfig}% Support for small, `sub' figures and tables

\usepackage[numbers,sort&compress]{natbib}% Citation support using natbib.sty
\bibpunct[, ]{[}{]}{,}{n}{,}{,}% Citation support using natbib.sty
\renewcommand\bibfont{\fontsize{10}{12}\selectfont}% Bibliography support using natbib.sty
\makeatletter% @ becomes a letter
\def\NAT@def@citea{\def\@citea{\NAT@separator}}% Suppress spaces between citations using natbib.sty
\makeatother% @ becomes a symbol again

\theoremstyle{plain}% Theorem-like structures provided by amsthm.sty
\newtheorem{theorem}{Theorem}[section]
\newtheorem{lemma}[theorem]{Lemma}

\theoremstyle{definition}

\theoremstyle{remark}

\def\Sm{\mathcal{A}_\Theta}
\def\D{\mathcal{D}(\mathcal{I})}

\begin{document}

\articletype{}% Specify the article type or omit as appropriate

\title{A new closed-form expression for the solution of ODEs in a ring of distributions and its connection with the matrix algebra}

\author{
\name{S. Pozza\textsuperscript{a}\thanks{CONTACT S. Pozza. Email: pozza@karlin.mff.cuni.cz}}
\affil{\textsuperscript{a}Charles University, Sokolovská 83 186, 75 Prague 8, Czech Republic}
}

\maketitle

\begin{abstract}
A new expression for solving homogeneous linear ODEs based on a generalization of the Volterra composition was recently introduced. 
In this work, we extend such an expression, showing that it corresponds to inverting an infinite matrix. This is done by studying a particular subring and connecting it with a subalgebra of infinite matrices.
\end{abstract}


\section{Introduction}
Let $\tilde{A}(t)$ be an $N\times N$ matrix-valued function analytic over $t\in \mathcal{I}=[0,1]$ and $I_N$ the $N \times N$ identity matrix. Then, the system of ODEs
\begin{equation}\label{eq:ode:intro:hom}
    \frac{d}{dt}U_s(t) = \tilde{A}(t) U_s(t), \quad U_s(s)=I_N, \quad \text{ for } t \geq s, \quad t,s\in \mathcal{I},
\end{equation}
has a unique solution $U_s(t)$. 
When $\tilde{A}(\tau_1)\tilde{A}(\tau_2)=\tilde{A}(\tau_2)\tilde{A}(\tau_1)$ for every $\tau_1,\tau_2 \in \mathcal{I}$, $U_s(t)$ takes the form
$$U_s(t)=\exp\left(\int_s^{t} \tilde{A}(\tau)\, \text{d}\tau\right).$$ 
In general, however, $U_s(t)$ has no known simple expression in terms of $\tilde{A}(t)$.
Indeed, despite systems of non-autonomous linear ODEs are crucial, common problems that appear in a variety of contexts \cite{Autler1955,BenEtAll17,Blanes15,kwaSiv72,Lauder1986,Shirley1965},
their solution is surprisingly difficult to formulate by an analytic expression.

When $\tilde{A}(t)$ is a scalar function, \cite{PozVan22proc_scalar} shows that a closed form of the solution exists in the non-commutative ring $\mathcal{S}$ composed of a certain distribution set $\D$ \cite{schwartz1978}, the so-called $\star$-product \cite{ProceedingsPaper2020}, and the usual addition. The $\star$-product is a convolution-like operation that generalizes the \emph{Volterra composition} (e.g., \cite{Volterra1928}). The closed form is given in terms of a $\star$-product inverse in the ring. Moreover, it is easy to define a $\mathcal{S}$-module of matrices with a bilinear product that generalizes the results to the case of a matrix-valued $\tilde{A}(t)$; see \cite{Giscard2015,GiscardPozza2021}. In a few words, in this case, the solution $U_s(t)$ is given by the bilinear product inverse of a matrix in the $\mathcal{S}$-module. This means that in the framework of the $\mathcal{S}$ ring, it is possible to express $U_s(t)$ in a closed form for every matrix-valued analytic function $\tilde{A}(t,s)$. This new expression has led to several new symbolic and numerical approaches to the solution of \eqref{eq:ode:intro:hom} \cite{Giscard2015,BonGis20,ProceedingsPaper2020,GiscardPozza2021,GiscardPozza2022,PozVan22proc_mtx,PozVan22proc_PANM,PozVan22proc_scalar}.

In the pieces of literature mentioned above, the new expression for the solution of \eqref{eq:ode:intro:hom} has not been derived in the ring $\mathcal{S}$, but in alternative equivalent ways. This paper aims to show the potentiality of working in the $\mathcal{S}$-ring module. We do that by deriving a new result, namely, the expression for the solution of the non-homogeneous system of linear ODEs
\begin{equation}\label{eq:ode:intro} 
\tilde{A}(t) U_s(t)=\frac{d}{dt}U_s(t) + \tilde{B}(t), \quad U_s(s)=I_N, \quad \text{ for } t \geq s, \quad t,s\in \mathcal{I},
\end{equation}
where $\tilde{B}(t)$ is an $N\times N$ matrix-valued analytic function over $\mathcal{I}$. 
Moreover, we will show that there is a subring of $\mathcal{S}$ that corresponds to a subalgebra of infinite matrices, and we will prove the existence of certain matrix inverses in the subalgebra using the connection with $\mathcal{S}$.

In Section \ref{sec:starexpression}, we define the $\star$-product and the related algebraic structures, and we derive the new expression of the solution of \eqref{eq:ode:intro}. Section \ref{sec:starsol} shows the connection between the $\mathcal{S}$ subring and a subalgebra of infinite matrices. As a consequence, the ODE solution can be obtained by solving a linear system in the subalgebra. 
Section \ref{sec:conclusion} concludes the presentation.

\section{A $\star$-product solution to non-homogeneous ODEs}\label{sec:starexpression}
Let $\tilde{f}_1(t,s), \tilde{f}_2(t,s)$ be two bivariate functions and assume that they are analytic\footnote{Note that in the previously appeared works, we have usually assumed the functions to be smooth. Here we restrict the assumption to analytic for the sake of simplicity.}, in both $t$ and $s$, over $\mathcal{I}=[0, 1]$; we denote such a set of functions by $\mathcal{A}(\mathcal{I})$.
The \emph{Volterra composition} of $\tilde{f}_1, \tilde{f}_2$, introduced by Vito Volterra (e.g., \cite{Volterra1928}), is defined as
\begin{equation*}
  \big(\tilde{f}_2 \star_v \tilde{f}_1\big)(t,s) := \int_s^{t} \tilde{f}_2(t,\tau) \tilde{f}_1(\tau, s) \, \text{d}\tau, \quad t,s \in \mathcal{I}.
\end{equation*}
Note that, from now on, a function marked with a tilde will stand for a function from $\mathcal{A}(\mathcal{I})$.
If we look at it as a product, the Volterra composition lacks important features. For instance, the identity. This is why the Volterra composition has been extended to the so-called $\star$-product \cite{ProceedingsPaper2020}.
Let $\Theta(t-s)$ be the Heaviside theta function, i.e.,
\begin{equation*}
    \Theta(t-s) = \begin{cases}
                        1, \quad t \geq s \\
                        0, \quad t < s 
                \end{cases}.
\end{equation*}
Moreover, let $\delta(\cdot)=\delta^{(0)}(\cdot)$ be the Dirac delta distribution and $\delta^{(i)}(\cdot)$ be its $i$th derivatives.
We denote with $\D$ the class of the distributions $d$ that can be expressed as
\begin{equation*}
d(t,s)=\widetilde{d}(t,s)\Theta(t-s) + \sum_{i=0}^k \widetilde{d}_i(t,s)\delta^{(i)}(t-s),  
\end{equation*}
with $\tilde{d}, \tilde{d}_i \in \mathcal{A}(\mathcal{I})$.
The $\star$-product $ \star: \D \times \D \rightarrow \D $ is defined as
\begin{equation}\label{eq:def:star}
  \big(f_2 \star f_1\big)(t,s) := \int_\mathcal{I} f_2(t,\tau) f_1(\tau, s) \, \text{d}\tau, \quad f_1, f_2 \in \D,
\end{equation}
Consider the subclass $\Sm(\mathcal{I}) \subset \D$ comprising those distributions of the form
\begin{equation*}
f(t,s)=\widetilde{f}(t,s)\Theta(t-s).
\end{equation*}
Then, the $\star$-product of $f_1,f_2 \in \Sm(\mathcal{I})$ is equivalent to the Volterra composition
\begin{align*}
  \big(f_2 \star f_1\big)(t,s) &= \int_{\mathcal{I}} \widetilde{f}_2(t,\tau) \widetilde{f}_1(\tau, s)\Theta(t-\tau)\Theta(\tau-s) \, \text{d}\tau,\\ &=\Theta(t-s)\int_s^{t} \widetilde{f}_2(t,\tau) \widetilde{f}_1(\tau, s) \, \text{d}\tau = \Theta(t-s)( \tilde{f}_2 \star_v \tilde{f}_1)(t,s).
\end{align*}
The $\star$-product  is well-defined and closed in $\D$; we refer the reader to \cite{GiscardPozza2021,ProceedingsPaper2020} for further details. For the goals of this paper, it will be enough to recall the following properties. Given $f \in \Sm(\mathcal{I})$, then
\begin{align}
  \left(\delta'(t-s) \ast f \right) (t,s)  &= \left(\partial_t \tilde{f}(t,s)\right)\Theta(t-s) + \tilde{f}(s,s)\delta(t-s); \label{eq:diracder} \\ 
    \left(f\ast\delta'(t-s) \right) (t,s) &= -\left(\partial_s \tilde{f}(t,s)\right)\Theta(t-s) + \tilde{f}(t,t)\delta(t-s); \nonumber
\end{align}
see \cite{schwartz1978,ProceedingsPaper2020}. As a consequence,
\begin{equation}\label{eq:dira1:inverse}
    \Theta(t-s) \star \delta'(t-s) = \delta'(t-s) \star \Theta(t-s) = \delta(t-s),
\end{equation}
i.e., $\delta'$ is the $\star$-inverse of $\Theta$.
Moreover,
\begin{itemize}
    \item $\D$ is closed under $\star$-multiplication;
    \item the $\star$-product is associative over $\D$;
    \item the Dirac delta distribution $1_\star(t,s):=\delta(t-s)$ is the identity of the $\star$-product.
\end{itemize}
Therefore, $\mathcal{S}(\mathcal{I}):=(\D, \star, +, 0, 1_\star)$ is a non-commutative ring.
The $\star$-product can also be extended to matrices and vectors composed of elements from $\D$. This is easily done by replacing the standard multiplication appearing in the integrand of \eqref{eq:def:star} with the usual matrix-matrix multiplication \cite{GiscardPozza2022}. Similarly, we can define the right (and left) scalar-matrix multiplication. As a result, we obtain the module of the matrices with elements from $\D$ with, as bilinear product, the $\star$-product between matrices, and, as scalar product, the $\star$-product between a scalar and a matrix.

\bigskip

The system of ODEs in \eqref{eq:ode:intro} can be rewritten in the form
\begin{equation}\label{eq:ode:star}
    \partial_t U(t,s) = \tilde{A}(t) \Theta(t-s) U(t,s) + \tilde{B}(t,s)\Theta(t-s), \quad U(s,s) = I_N, \quad t,s \in \mathcal{I},
\end{equation}
with $U_s(t) = U(t,s)$. 
Note that the matrices $U(t,s) = \tilde{U}(t,s)\Theta(t-s)$, $A(t,s) := \tilde{A}(t) \Theta(t-s)$, and $B(t,s):= \tilde{B}(t,s) \Theta(t-s)$ are all composed of elements from $\Sm(\mathcal{I})$.
Therefore, by exploiting formula \eqref{eq:diracder} and \eqref{eq:dira1:inverse}, equation~\eqref{eq:ode:star} becomes:
\begin{equation}\label{eq:ode:star2}
    \delta'(t-s) \star U(t,s) = \tilde{A}(t)U(t,s) + I_\star(t,s) + B(t,s),
\end{equation}
where $I_\star(t,s)$ is the identity matrix $I_N$ multiplied by $\delta(t-s)= 1_\star(t,s)$.
Once the problem has been rewritten into the $\star$-framework, we can derive a formula for its solution by working in the $\mathcal{S}(\mathcal{I})$-module. If we $\star$-multiplying \eqref{eq:ode:star2} from the left by $\Theta(t-s)$ we obtain
\begin{equation}\label{eq:star:u:iter}
    U(t,s) = \Theta(t-s) \star \left( \tilde{A}(t)U(t,s) + I_\star(t,s) + B(t,s) \right).
\end{equation}
Now, by replacing $U(t,s)$ in the right-hand side of \eqref{eq:star:u:iter} with the right-hand side of \eqref{eq:star:u:iter} itself, we get the following iterations (we drop the dependency from $t,s$ for the sake of readability)
\begin{align*}
      U &= \Theta \star \left( \tilde{A}U + I_\star + B \right), \\
      &= \Theta \star \left( \tilde{A}\left(\Theta \star \left( \tilde{A}U + I_\star + B \right)\right) + I_\star + B \right), \\
      &= \Theta \star \left( A \star \left( \tilde{A}U + I_\star + B \right) + I_\star + B \right), \\
      &= \Theta \star \left( A \star \tilde{A}U + \left( A + I_\star \right) \star ( I_\star + B ) \right).
\end{align*}
Note that the equality
\begin{equation*}
    \tilde{A}\left(\Theta \star \left( \tilde{A}U + I_\star + B \right)\right) = A \star \left( \tilde{A}U + I_\star + B \right)
\end{equation*}
holds since $\tilde{A}(t)$ does not depend on $s$.
Repeating the iterations $k$ times\footnote{In fact, such iterations are Picard iterations, see \cite[Section 2]{PozVan22proc_PANM}.},  we obtain
\begin{align}
      U &= \Theta \star \left( A \star A \star \tilde{A}U + \left( A \star A + A + I_\star \right) \star ( I_\star + B ) \right), \nonumber \\ 
       & \, \, \,\vdots \nonumber \\ 
      &= \Theta \star \left( A^{k\star} \star \tilde{A}U + \left( A^{k\star} + \dots + A + I_\star \right) \star ( I_\star + B ) \right), \label{eq:picard:it}
\end{align}
with $A^{k\star}$ the $k$th $\star$-power of $A$.
As shown in \cite{GiscardPozza2022}, 
\begin{equation*}
    \max_{t,s \in \mathcal{I}} \left\| \left(A^{k\star}\right)(t,s) \right\| \leq \left(\max_{t,s \in \mathcal{I}}\|A(t,s) \|\right)^k \frac{(t-s)^{k-1}}{(k-1)!}, \quad k \geq 1,
\end{equation*}
for any induced matrix norm.
Therefore \eqref{eq:picard:it} uniformly converges to the expression
\begin{equation}\label{eq:star:sol}
    U(t,s) = \Theta(t-s) \star R_\star(A)(t,s) \star \left( I_\star(t,s) + B(t,s) \right),
\end{equation}
where $R_\star(A)$ is the $\star$-resolvent of $A$, i.e.,
\begin{equation*}
    R_\star(A) = I_\star + \sum_{k=1}^\infty \left(A^{\star k}\right)(t,s).
\end{equation*}
Noticing that
 \begin{equation}\label{eq:res:inv}
     R_{\star}(A) \star (I_\star - A) = \left(I_\star + \sum_{k\geq 1} A^{\star k}\right) \star (I_\star - A) = I_\star,
 \end{equation}
 that is, $R_{\star}(A) = (I_\star - A)^{-\star}$ (the $\star$-inverse of $(I_\star - A)$), we get 
 \begin{equation}\label{eq:star:sol:closed}
    U(t,s) = \Theta(t-s) \star (I_\star - A)^{-\star}(t,s) \star \left( I_\star(t,s) + B(t,s) \right),
\end{equation}
that is a closed-form expression in the $\mathcal{S}$-module.

Finally, since the matrix $A$ is composed of elements from the subset 
$$\mathcal{A}_\Theta^t(\mathcal{I}) :=\left\{f \in \Sm(\mathcal{I}) : f(t,s) = \tilde{f}(t)\Theta(t-s) \right\} \subset \Sm(\mathcal{I}),$$ 
it is useful to define the set
$$ \mathcal{D}^t_0(\mathcal{I}) := \left\{f(t,s) = \alpha 1_\star + \sum_{i=1}^n \big(g_{i,1} \star \dots \star g_{i,m_n} \big), \; g_{i,j} \in \mathcal{A}_\Theta^t(\mathcal{I}) , \; \alpha \in \mathbb{C}\right\}. $$
and the related subring $(\mathcal{D}^t_0(\mathcal{I}), \star, +, 0, 1_\star)$.

\section{The $\star$-product and the matrix algebra}\label{sec:starsol}
 Let $\{p_k\}_k$ be a sequence of orthonormal shifted Legendre polynomials over the bounded interval $\mathcal{I}=[0,1]$, i.e., 
		\begin{align*}
			\int_{\mathcal{I}} p_k(\tau) p_\ell(\tau) d\tau = \delta_{k,\ell} = \begin{cases}
				0,\quad \text{if }k\neq \ell\\
				1,\quad \text{if } k=\ell				
			\end{cases}.
		\end{align*}
		Despite the fact that the functions $p_k$ are not in $\D$, with a small abuse of notation, we can still define the product
        $$ p_k(s) \star p_\ell(t) = \int_\mathcal{I} p_k(\tau) p_\ell(\tau) \; \textrm{d} \tau = \delta_{k,\ell}.  $$
		Given a function $f(t,s)= \tilde{f}(t,s) \Theta(t-s) \in \Sm(\mathcal{I})$, we can expand it into the series
        \begin{equation}\label{eq:f:exp}
			f(t,s) = \sum_{k=0}^\infty \sum_{\ell=0}^\infty f_{k,\ell} \, p_k(t) p_\ell(s), \quad t \neq s, \quad t,s \in \mathcal{I},
		\end{equation}
		with coefficients
		\begin{equation*}
		    f_{k,\ell} = \int_\mathcal{I} \int_\mathcal{I} f(\tau,\rho) p_k(\tau) p_\ell(\rho) \; \textrm{d} \rho \; \textrm{d} \tau;
		\end{equation*}
		see, e.g., \cite[p.~55]{LebSil72}.  
  The expansion can be rewritten in the matrix form
	\begin{align*}
	f(t,s) = \sum_{k=0}^{\infty} \sum_{\ell=0}^{\infty} f_{k,\ell} \, p_k(t) p_\ell(s) = \phi(t)^T F \, \phi(s)  \quad t \neq s, \quad t,s \in \mathcal{I}.
		\end{align*}
		where the \emph{coefficient matrix} $F$ and the vector $\phi(\tau)$ are defined as follows
		\begin{equation}\label{eq:coeff:mtx}
		          F := \begin{bmatrix}
				f_{0,0} & f_{0,1} & f_{0,2} & \dots \\
				f_{1,0} & f_{1,1} & f_{1,2} & \dots \\
    		      f_{2,0} & f_{2,1} & f_{2,2} & \dots \\
				\vdots & \vdots   & \vdots  &  \ddots
		\end{bmatrix}, 
		\quad  \phi_M(\tau) :=
		 \begin{bmatrix}
			p_0(\tau)\\
			p_1(\tau)\\
   			p_2(\tau) \\
			\vdots
		\end{bmatrix}.
		\end{equation}
  Note that each element of $F$ can be bounded by
\begin{equation}\label{eq:F:bound:const}
    |f_{k,\ell}| \leq \max_{t, s \in [0,1]} |\tilde{f}(t,s)| \sqrt{2k + 1} \sqrt{2\ell + 1}
\end{equation}

  In particular, an univariate function $\tilde{f}(t)$ can be expanded as 
	\begin{equation*}
		\tilde{f}(t) := \sum_{d=0}^\infty \alpha_d p_d(t),\quad \text{with } \alpha_d = \int_{-1}^1 \tilde{f}(t) p_d(t) dt.
	\end{equation*}
  Let $B^{(k)}$ be the coefficient matrix of $p_k(t) \Theta(t-s)$. Then, the coefficient matrix of $f(t,s) = \tilde{f}(t)\Theta(t-s)$ can be also expanded into the series
$$ F = \sum_{k=0}^\infty \alpha_k B^{(k)}.  $$
Note that each $B^{(k)}$ is a banded matrix with bandwidth $k+1$, \cite{PozVan22proc_scalar}.
Moreover, the Fourier coefficients $\{\alpha_k\}_{k\geq 0}$ decay geometrically \cite{TreApp13}. Indeed, there exist $0<\rho<1, C>0$ such that
	\begin{equation}\label{eq:coeffDecayRate}
		\vert a_k\vert \leq C \rho^{k}.
	\end{equation}
As a consequence, each element of $F$ can be bounded as follows
\begin{align}
        |f_{k,\ell}| &= \left| \sum_{j=|k-\ell|+2}^\infty \alpha_j B^{(j)}_{k,\ell} \right | \leq \sum_{j=|k-\ell|+2}^\infty |\alpha_j| | B^{(j)}_{k,\ell} | \nonumber \\
        &\leq C \max_{t, s \in [0,1]} |\tilde{f}(t,s)| \sum_{j=|k-\ell|+2}^\infty \rho^{j} \sqrt{2k + 1} \sqrt{2\ell + 1} \nonumber \\
                    &\leq K \rho^{|k-\ell|+2}, \label{eq:F:bound}
\end{align}
for some $K>0$. This means that $F$ is characterized by a geometric decay of the element magnitude as we move away from the diagonal.
  
  Consider $f(t,s) = \tilde{f}(t)\Theta(t-s),g(t,s) = \tilde{g}(t)\Theta(t-s) \in \Sm^t(\mathcal{I})$ and the related coefficient matrices \eqref{eq:coeff:mtx} $F,G$, respectively. 
  By orthogonality, $\phi(s) \star \phi(t)^T = I$, with $I$ the identity matrix. Therefore, for every $t \neq s$,
  		\begin{align*}
			(f\star g)(t,s) &= \left( \phi(t)^T F \, \phi(s) \right) \star \left( \phi(t)^T G \, \phi(s) \right), \\
                     &= \phi(t)^T F \left(\phi(s) \star \phi(t)^T \right) G \, \phi(s), \\
                    &= \phi(t)^T FG \, \phi(s).
		\end{align*}
  Thus, the coefficient matrix $H$ of the function $h = f \star g \in \mathcal{D}_0^t(\mathcal{I})$ is given by the matrix-matrix product $FG$.
  It is important to note that the product $FG$ is well-defined. Indeed, the series is convergent since by \eqref{eq:F:bound} there exist $K>0$ and $0< \rho < 1$ so that
  \begin{align*}
        \left|(FG)_{k,\ell} \right| &= \left|\sum_{j=1}^\infty F_{k,j} G_{j,\ell} \right| \leq  \sum_{j=1}^\infty \left| F_{k,j} \right| \left| G_{j,\ell} \right|  \\
                    &\leq \sum_{j=1}^\infty K_f K_g \rho_f^{|k-j|+2}  \rho_g^{|\ell-j|+2} \leq \sum_{j=1}^\infty K \rho^{|k-j| + |\ell-j| + 4}.
  \end{align*}
    Moreover, since $\min_{j=1,2\dots} |k-j| + |\ell-j| = |k-\ell|$, there exist $K_{fg}>0$ and $0<\rho_{fg}<1$ so that
   \begin{align}\label{eq:FG:bound}
        \left|(FG)_{k,\ell} \right| \leq K_{fg} \rho_{fg}^{|k-\ell|}.
  \end{align}
    This latter bound show that $FG$ is also characterized by a geometric decay of the element magnitude as we move away from the diagonal. Therefore, given $F, G, H$ coefficient matrices of $f(t,s) = \tilde{f}(t)\Theta(t-s), g(t,s) = \tilde{g}(t)\Theta(t-s), h(t,s) = \tilde{h}(t)\Theta(t-s)$, the matrix product $FGH$ is also well-defined and characterized by a off-diagonal geometric decay. 
    As a consequence, the set $\mathcal{F}$ of all the coefficient matrices of functions from $\mathcal{D}_0^t(\mathcal{I})$ is a subalgebra (with the usual sum, product, and matrix product) and it corresponds to the subring $(\mathcal{D}_0^t(\mathcal{I}), \star, +, 0, 1_\star)$.

\bigskip

  Consider now the $N \times N$ matrix-valued functions $A(t,s) = [f_{i,j}(t,s)]_{i,j=1}^N, B(t,s) = [g_{i,j}(t,s)]_{i,j=1}^N \in \mathbb{C}$ composed of elements from $\mathcal{D}_0^t(\mathcal{I})$. The functions $f_{i,j}$ and $g_{i,j}$ are associated with their coefficient matrices $F^{(i,j)}, G^{(i,j)}$, respectively. By extending the arguments presented above, we get the following expression for the (matrix) $\star$-product $C = A \star B = [h_{i,j}]_{i,j=1}^N$,
  \begin{equation*}
      \left(A \star B \right)_{k,\ell}(t,s) = \sum_{j=1}^N (f_{k,j} \star g_{j,\ell})(t,s) =  \sum_{j=1}^N \phi(t)^T F^{(k,j)} G^{(j,\ell)} \, \phi(s), \quad t\neq s.
  \end{equation*}
  Therefore, the coefficient matrices $H^{(k,\ell)}$ of the functions $h_{k,\ell}(t,s)$ are given by
  \begin{equation} \label{eq:block:prod}
        H^{(k,\ell)} = \sum_{j=1}^N F^{(k,j)} G^{(j,\ell)}.
  \end{equation}
  Defining the block matrices $\mathbf{A} = [F_{i,j}]_{i,j=1}^N$, $\mathbf{B} = [G_{i,j}]_{i,j=1}^N$, $\mathbf{C} = [H_{i,j}]_{i,j=1}^N$, we obtain the relation: 
  \begin{equation}\label{eq:block:mtx:prod}
      \mathbf{C} = \mathbf{A} \mathbf{B}.
  \end{equation}
  Note that, despite the blocks having an infinite size, the product in \eqref{eq:block:mtx:prod} is well-defined since the matrix products in \eqref{eq:block:prod} are well-defined.

In Section~\ref{sec:starexpression}, we have seen the crucial role played by the $\star$-resolvent $R_\star(A)$.
Let $C := \sum_{k\geq 1} A^{\star k} = [h_{i,j}]_{i,j=1}^N$, with $ h_{i,j} \in \mathcal{A}_\Theta$. Therefore, using the notation above, 
		\begin{align*}
			\left(R_{\star}(A)\right)_{k,\ell}(t,s) &= \phi(t)^T \left(I + H^{(k,\ell)} \right) \phi(s).
		\end{align*}
  Hence, for $t \neq s$, we obtain
		\begin{align}\label{eq:reso:mtx:1}
			\left(R_{\star}(A)\star (I_\star - A) \right)_{k,\ell} &= 
            \sum_{j=1}^N \phi(t)^T 
            \left(I + H^{(k,j)} \right) 
            \phi(s)
                \star
            \phi(t)^T 
            \left(I - F^{(j,\ell)} \right) 
            \phi(s), \\ \label{eq:reso:mtx:2}
            &= 
            \phi(t)^T 
            \sum_{j=1}^N 
            \left(I + H^{(k,j)} \right) 
            \left(I - F^{(j,\ell)} \right) 
            \phi(s) .
		\end{align}
  Then, relation \eqref{eq:res:inv} implies
		\begin{align}\label{eq:reso:mtx:3}
                        \phi(t)^T 
            \sum_{j=1}^N 
            \left(I + H^{(k,j)} \right) 
            \left(I - F^{(j,\ell)} \right) 
            \phi(s)  
            = 1_\star \delta_{k,\ell} =  \phi(t)^T I\, \phi(s) \delta_{k,\ell}.
		\end{align}
In the following, we prove that $\sum_{j=1}^N 
            \left(I + H^{(k,j)} \right) 
            \left(I - F^{(j,\ell)} \right)  = I$.
    \begin{lemma}\label{lemma:dirac:ident}
    Let $D$ be an infinite matrix so that $\phi(t)^T D \phi(s) = \delta(t-s)$, then $D = I$.    
    \end{lemma}
\begin{proof}
        First of all, since $\delta(t-s)$ is a generalized function, the convergence of the series $\phi(t)^T D \phi(s)$ is intended in a weak sense. This means that, for every $\tilde{f}(t)$ analytic on $\mathcal{I}$,
        \begin{equation*}
            \lim_{N\rightarrow \infty}\int_{\mathcal{I}} \sum_{k,\ell=1}^N d_{k,\ell} p_k(\tau) p_\ell(s) \tilde{f}(\tau) d\tau = \tilde{f}(s) = \int_{\mathcal{I}} \delta(\tau - s)\tilde{f}(\tau) d\tau.
        \end{equation*}
        Setting $\Tilde{f}(t) = p_j(t)$ gives
        \begin{align*}
         p_j(s) &= \lim_{N\rightarrow \infty} \sum_{k,\ell=1}^N d_{k,\ell} \int_{\mathcal{I}} p_k(\tau) p_\ell(s) p_j(\tau) d\tau \\
         &= \lim_{N\rightarrow \infty} \sum_{k,\ell=1}^N d_{k,\ell} p_\ell(s) \int_{\mathcal{I}} p_k(\tau) p_j(\tau) d\tau, 
                = \sum_{\ell=1}^\infty d_{j,\ell} p_\ell(s).
        \end{align*}
        As the Legendre expansion of $p_j(s)$ is unique, $d_{j,\ell} = \delta_{j,\ell}$, for $j,\ell = 1, 2, \dots$ . 
\end{proof}

\begin{theorem}\label{thm:res:exist}
   Let $A(t,s)$ be an $N \times N$ matrix-valued function composed of elements from $\mathcal{D}_0^t(\mathcal{I})$ and let $\mathbf{A}=[F^{(k,\ell)}]_{k,\ell=1}^N$ be the related block matrix, with $F^{(k,\ell)}$ the coefficient matrix of $A_{k,\ell}(t,s)$.
   Moreover, consider the matrix-valued function $C(t,s) = \sum_{k \geq 1} A^{\star k}(t,s)$, and let $\mathbf{C}=[H^{(k,\ell)}]_{k,\ell=1}^N$ be the related block matrix, with $H^{(k,\ell)}$ the coefficient matrix of $C_{k,\ell}(t,s)$. 
   Then 
   \begin{equation*}
       (I + \mathbf{C})(I- \mathbf{A}) = I,
   \end{equation*}
   that is, $(I- \mathbf{A})$ is invertible.
\end{theorem}
\begin{proof}
        Equations~\eqref{eq:reso:mtx:1}, \eqref{eq:reso:mtx:2} and \eqref{eq:reso:mtx:3} show that 
        \begin{align*}
           \phi(t)^T \sum_{j=1}^N \left(I + H^{(k,j)} \right) \left(I - F^{(j,\ell)} \right) \phi(s)  
           = \phi(t)^T D^{(k,\ell)} \phi(s) = \delta(t-s) \delta_{k,\ell},
        \end{align*}
        where the products $ (I + H^{(k,j)} ) (I - F^{(j,\ell)} )$ are well-defined since $(I + H^{(k,j)} )_{m,n}$ is bounded by $1+K_h \sqrt{2m+1}\sqrt{2n+1}$ and $(I - F^{(j,\ell)} )_{m,n}$ by $K_f \rho^{|k-\ell|}$, see \eqref{eq:F:bound:const} and \eqref{eq:FG:bound}.
        Thus, for $k = \ell$, $D^{(k,k)} = \sum_{j=1}^N \left(I + H^{(k,j)} \right) \left(I - F^{(j,\ell)} \right) = I$ by Lemma~\ref{lemma:dirac:ident}. For $k \neq \ell$, $D^{(k,\ell)} = 0$.
\end{proof}

   Consider the solution $U(t,s)$ of \eqref{eq:ode:star} and let us define $\mathbf{T}$ as the block diagonal matrix with blocks all equal to $T$, the coefficient matrix of $\Theta(t-s)$. Using the notation of Theorem~\ref{thm:res:exist} and Expression~\eqref{eq:star:sol}, 
  we can transform the ODE \eqref{eq:ode:star} into the matrix problem:
    \begin{equation}\label{eq:inf:mtx:sol}
         \mathbf{U} = \mathbf{T} (I + \mathbf{C}) (I + \mathbf{B}) = \mathbf{T} (I - \mathbf{A})^{-1} (I + \mathbf{B}),
    \end{equation}
    where the block matrix $\mathbf{B} = [G^{(k,\ell)}]_{k,\ell=1}^N$ is composed of the coefficient matrices $G^{(k,\ell)}$ of the functions $B_{k,\ell}(t,s)$.
   Hence, the solution of \eqref{eq:ode:star} can be expressed by
   \begin{equation*}
       U_{k,\ell}(t,s) = \phi(t)^T Y^{(k, \ell)} \, \phi(s), \quad k,\ell =1,\dots, N,
   \end{equation*}
   with $\mathbf{U} = [Y^{(k,\ell)}]_{k,\ell=1}^N$.
  
 To conclude the presentation, we need to discuss the convergence of the expansion \eqref{eq:f:exp}. Indeed, since $f$ is discontinuous for $t=s$, the expansion does not converge to $f(t,t)$ and, moreover, it converges only linearly for $t\neq s$; see, e.g., \cite{LebSil72,TreApp13}.
 \begin{lemma}\label{eq:conv:disc}
        Consider $f(t,s) \in \mathcal{A}_\Theta$ and the related expansion in orthonormal shifted Legendre polynomials \eqref{eq:f:exp}. Then,
        \begin{equation*}
        \lim_{N \rightarrow \infty} 	\sum_{k=0}^N \sum_{\ell=0}^N f_{k,\ell} \, p_k(t) p_\ell(s) = \left\{ \begin{array}{lc}
            f(t,s), & t\neq s  \\
            f(t,t)/2, & t =s
        \end{array} \right. , \quad t,s \in (0,1).
                \end{equation*}
 \end{lemma}
 \begin{proof}
        The proof is direct consequence of Theorem~1 and Remark~1 in Section~4.7 of \cite{LebSil72}.
 \end{proof}
 
  However, for the fixed $s = 0$, the univariate function $f(t,0) = \tilde{f}(t,0)\Theta(t-0) = \tilde{f}(t,0)$ is analytic over $[0,1]$. Therefore,
  defining $a_k = \sum_{\ell=0}^\infty (f_{k,\ell}\, p_\ell(0))$, we get the Legendre expansion
  \begin{equation*}
      	f(t,0) = \sum_{k=0}^\infty  p_k(t) \sum_{\ell=0}^\infty f_{k,\ell} \, p_\ell(0) = \sum_{k=0}^\infty a_k p_k(t), \quad t \in [0,1].
  \end{equation*}
  Therefore, the truncated series $\sum_{k=0}^M a_k p_k(t)$ converges geometrically to $f(t,0)$.
  As a consequence, in a numerical setting, we can approximate the function $f(t,0)$ by using $F_M$, the principal leading submatrix of $F$, obtaining the approximation
  \begin{equation*}
      f(t,0) \approx \phi_M(t)^T F_M \, \phi_M(0),
  \end{equation*}
  with $\phi_M(t)$ the first $M$ elements of $\phi(t)$.
  In this case, we expect to reach a good enough accuracy for a (relatively) small $M$. Note that for $s>0$ this is not possible, as we expect the emergence of the Gibbs phenomenon; see, e.g., \cite{TreApp13}.

  By considering the principal leading submatrix of each of the blocks in formula \eqref{eq:inf:mtx:sol}, for $s=0$, we get the following approximated solution to \eqref{eq:ode:intro} 
   \begin{equation}\label{eq:ode:trunc}
       U_0(t) \approx (I_N \otimes \phi_M(t)^T T_M) (I_M- \mathbf{A}_M)^{-1} (I_M + \mathbf{B}_M) (I_N \otimes \phi_M(0));
    \end{equation}
see also \cite{PozVan22proc_mtx}.
  The numerical approach for the solution of a non-autonomous linear ODE system derived from \eqref{eq:ode:trunc} can be found in \cite{PozVan22proc_mtx,PozVan22proc_scalar,PozVan22proc_PANM} where several numerical examples show its efficacy. 
  
\section{Conclusion}\label{sec:conclusion}
In this paper, we have presented a new expression for the solution of a (non-homogeneous, non-autonomous) system of linear ODEs using the so-called $\star$-product. The $\star$-product, the usual sum, and a specific set of distributions constitute a ring $\mathcal{S}$. We have also shown that a certain subring of $\mathcal{S}$ corresponds to a subalgebra of infinite matrices. Thanks to this correspondence, we have expressed the solution of the linear ODE system in the infinite matrix algebra. Such a solution is obtained by inverting a determined infinite matrix.
The connection between the $\star$-product ring and the matrix subalgebra helped us to show that such an inverse always exists. By truncating the infinite matrices, it is possible to derive numerical methods for the solution of ODEs. This paper complements the results we are developing in the truncated case by placing it in the general framework of the infinite matrix algebra.

\section*{Acknowledgements}
This  work  was  supported  by  Charles  University  Research programs UNCE/SCI/023 and PRIMUS/21/SCI/009 and by the Magica project ANR-20-CE29-0007 funded by the French National Research Agency.

\bibliographystyle{tfnlm}
\bibliography{biblio}

\end{document}